\documentclass[envcountsect, envcountsame]{svmult}

\usepackage{mathptmx}       
\usepackage{helvet}         
\usepackage{courier}        
\usepackage{type1cm}        
\usepackage{makeidx}         
\usepackage{graphicx}        
\usepackage{multicol}        
\usepackage[bottom]{footmisc}

\makeindex             

\usepackage{amstext,amsfonts,amssymb,amscd,amsbsy,amsmath}

\newcommand{\NN}{\mathbb{N}}
\newcommand{\PP}{\mathbb{P}}

\newcommand{\ZZ}{\mathbb{Z}}

\DeclareMathOperator{\Gr}{Gr}

\DeclareMathOperator{\trop}{trop}

\newcommand{\cO}{\mathcal{O}}

\smartqed

\begin{document}

\title*{Khovanskii Bases of Cox-Nagata Rings and Tropical Geometry}
\author{Martha Bernal, Daniel Corey, Maria Donten-Bury, Naoki Fujita, and Georg Merz}
\institute{Martha Bernal \at Unidad Acad\'emica de Matem\'aticas UAZ, Calzada Solidaridad, Zacatecas, Mexico, \email{m.m.bernal.guillen@gmail.com}
\and Daniel Corey \at Yale University, Department of Mathematics, \email{daniel.corey@yale.edu}
\and Maria Donten-Bury \at University of Warsaw, Institute of Mathematics, Banacha 2, 02-097 Warszawa, Poland, \email{m.donten@mimuw.edu.pl}
\and Naoki Fujita \at Department of Mathematics, Tokyo Institute of Technology, 2-12-1 Oh-okayama, Meguro-ku, Tokyo 152-8551, Japan, \email{fujita.n.ac@titech.ac.jp}
\and Georg Merz \at  Mathematisches Institut,
Georg-August Universit\"at G\"ottingen
Bunsenstra\ss e 3-5,
D-37073 G\"ottingen, Germany,
 \email{georg.merz@mathematik.uni-goettingen.de}
}

%
%
\maketitle

\abstract*{The Cox ring of a del Pezzo surface of degree 3 has a distinguished set of 27 minimal generators. We investigate conditions under which the initial forms of these generators generate the initial algebra of this Cox ring.  Sturmfels and Xu provide a classification in the case of degree 4 del Pezzo surfaces by subdividing the tropical Grassmannian $\operatorname{TGr}(2,\mathbb{Q}^5)$. After providing the necessary background on Cox-Nagata rings and Khovanskii bases, we review the classification obtained by Sturmfels and Xu. Then we describe our classification problem in the degree 3 case and its connections to tropical geometry. In particular, we show that two natural candidates, $\operatorname{TGr}(3,\mathbb{Q}^6)$ and the Naruki fan, are insufficient to carry out the classification. }

\abstract{The Cox ring of a del Pezzo surface of degree 3 has a distinguished set of 27 minimal generators. We investigate conditions under which the initial forms of these generators generate the initial algebra of this Cox ring.  Sturmfels and Xu provide a classification in the case of degree 4 del Pezzo surfaces by subdividing the tropical Grassmannian $\operatorname{TGr}(2,\mathbb{Q}^5)$. After providing the necessary background on Cox-Nagata rings and Khovanskii bases, we review the classification obtained by Sturmfels and Xu. Then we describe our classification problem in the degree 3 case and its connections to tropical geometry. In particular, we show that two natural candidates, $\operatorname{TGr}(3,\mathbb{Q}^6)$ and the Naruki fan, are insufficient to carry out the classification. }

\section{Introduction}

The starting point for this chapter is the following problem proposed by Sturmfels and Xu as \cite[Problem~5.4]{SturmfelsEtAl}: \emph{determine all equivalence classes of 3-dimensional sagbi subspaces of $\Bbbk^6$}. In the next few paragraphs we explain its statement in detail and give an outline of the chapter.

Let us begin with clarifying two important aspects of our notation. First, instead of the names \emph{sagbi bases} resp. \emph{sagbi subspaces} (where \emph{sagbi}, first used in~\cite{RobbianoSweedler}, stands for ``subalgebra analogue to Gr\"obner bases for ideals'') we will use the name \emph{Khovanskii bases} resp. \emph{ Khovanskii subspaces}. This new name was introduced in a much more general setting in a recent article~\cite{KavehManon}.

Second, we make some assumptions on the field~$\Bbbk$. We usually take~$\Bbbk$ to be the field of rational functions $\mathbb{Q}(t)$, but to formulate and work on this problem one may consider any other field with a nontrivial  valuation. The residue field of $\Bbbk$ for the considered valuation will be denoted by $k$.

The fundamental objects for this chapter are Khovanskii bases and moneric sets. We repeat their definitions after Sturmfels and Xu; see~\cite[Sect.~3]{SturmfelsEtAl} for more details and comments on their properties.
By $\operatorname{val}\colon \Bbbk^{\times}\to \mathbb{Z}$ we denote a valuation map of~$\Bbbk$. If $\Bbbk=F(t)$ for some field~$F$, we use the following valuation: $\operatorname{val(p)}\in \mathbb{Z}$ is the unique integer $\omega$ such that $t^{-\omega}p(t)$ takes a nonzero value at $t=0$.
Then, for $f \in \Bbbk[x_1,\ldots,x_n]$ we can compute its \emph{initial form} $\operatorname{in}(f)$. If $\omega_0$ is the minimum of $\operatorname{val}$ for coefficients of all monomials in~$f$, then
$$\operatorname{in}(f) = (t^{-\omega_0}f)|_{t=0} \in k[x_1,\ldots,x_n].$$
That is, $\operatorname{in}(f)$ identifies all monomials of $f$ whose coefficients have smallest valuation.

\begin{definition}
We call a subset $\mathcal{F} \subset \Bbbk[x_1,\ldots,x_n]$ \emph{moneric} if $\operatorname{in}(f)$ is a monomial for all $f \in \mathcal{F}$.
\end{definition}

For a $\Bbbk$-subalgebra $U \subseteq \Bbbk[x_1,\ldots,x_n]$ we define the \emph{initial algebra} $\operatorname{in}(U)$ as the $k$-subalgebra generated by $\operatorname{in}(f)$ for  $f \in U$.

\begin{definition}
  We say that a subset $\mathcal{F} \subset U$ is a \emph{Khovanskii basis} of a $\Bbbk$-subalgebra $U \subseteq \Bbbk[x_1,\ldots,x_n]$ if
  \begin{itemize}
  \item $\mathcal{F}$ is moneric, and
  \item the initial algebra $\operatorname{in}(U)$ is generated by $\{\operatorname{in}(f) | f \in \mathcal{F}\}$ as a $\Bbbk$-algebra.
  \end{itemize}
\end{definition}

We are interested in Khovanskii bases of Cox-Nagata rings, which will be described in Section~\ref{section_cox}. After they are introduced, we will be able to explain how a 3-dimensional subspace of $\Bbbk^6$ determines a basis, possibly a Khovanskii basis, of the Cox ring of a del Pezzo surface of degree~3. We say that such a subspace is moneric (resp. Khovanskii) if the corresponding basis is moneric (resp. Khovanskii), see Definition~\ref{def_Ksubsp}. We look at moneric subspaces up to an equivalence relation which respects the property of being a Khovanskii subspace, see Definition~\ref{def_equivalence}. 

We suggest that the reader treats this text as an introduction to the concept of Khovanskii bases and related research problems. For us, understanding the geometric motivation and connections was as important as solving the combinatorial classification problem itself. This is the reason why, besides presenting our approach to answering the main question, we also spend significant amount of time on exploring its background. 

In Section~\ref{section_cox} we define the Cox ring and explain its construction for del Pezzo surfaces. We also introduce the Nagata's action, which provides a link between linear subspaces of $\Bbbk^n$ and choices of initial forms of generators of Cox rings of del Pezzo surfaces (i.e. candidates for moneric or Khovanskii bases of the Cox ring).

Section~\ref{section_degeneration} is dedicated to explaining the geometric consequence of a Khovanskii basis in terms of degenerations.
Roughly speaking, a Khovanskii basis of a (finitely generated) subalgebra~$U$ of the polynomial ring yields a degeneration of~$\operatorname{Spec}(U)$ to a toric variety. We show that we obtain even more if we choose a Khovanskii basis of the Cox ring $\operatorname{Cox}(X)$ of a variety $X$: we do not only obtain a toric degeneration of $\operatorname{Spec}(\operatorname{Cox}(X))$, but also toric degenerations of $X$ with respect to all possible embeddings.

In Section~\ref{section_counting} we explain and give examples for the problem which motivated Sturmfels and Xu to study Khovanskii bases of Cox-Nagata rings. It turns out that a Khovanskii basis allows us to compute the Hilbert function of a del Pezzo surface with respect to a specific embedding by counting lattice points in dilations of a rational convex polytope.

Finally, Sections~\ref{section_trop} and~\ref{section_examples} describe our first attempts to classify 3-dimensional Khovanskii subspaces of $\Bbbk^6$. First we describe two tropical varieties which we expect to be related to the problem: the tropical Grassmannian $\operatorname{TGr}(3,6)$ and the tropical moduli space of del Pezzo surfaces of degree~3. Then we explain how we tried to use them as parametrizing spaces for moneric and Khovanskii subspaces. The conclusion is that neither of these models has the combinatorial structure suitable to play this role.

\section{Cox-Nagata Rings}\label{section_cox}

Let $G$ be a linear group acting on a polynomial ring $R$ over a field
$\Bbbk$. Hilbert's fourteenth problem asks whether the ring of invariants
$R^G$ is a finitely generated $K$-algebra. The answer is affirmative
when the group $G$ is reductive and also when $G={\mathbb
  G}_a$. Nagata considered the action of a codimension $3$ subspace $G\subset {\mathbb C}^n$  acting on $R={\mathbb C}[x_1,\dots,x_n,y_1,\dots,y_n]$ via
\[x_i\mapsto x_i \text { and } y_i\mapsto y_i +\lambda_ix_i,\] 
where $(\lambda_1,\dots, \lambda_n)\in G$.
He proved that the ring of invariants
$R^G$ is not finitely generated for $n=16$, see~\cite{Nagata}. Mukai realized the ring of
invariants $R^G$ as a certain Rees algebra and as such, it is isomorphic to
the Cox ring of a blow-up (\cite{Mukai}). Mukai's
description of $R^G$ provides conditions for it to be finitely
generated and a way of computing its generators, at least for
$\operatorname{codim} G\leq 3$.

In this section we review Mukai's description of $R^G$. We recall the definition of Cox rings and study with some more detail the isomorphism between $R^G$ when $\operatorname{codim} G=3$ and the Cox ring of the blow-up of $\PP^2$ at $n$ points in general position. Next we specialize to the blow-up of six points and give a description of the invariants that generate $R^G$.

\subsection{Nagata's action} Let $R={\Bbbk}[x_1,\dots,x_n,y_1,\dots,y_n]$ be an algebra with a $\ZZ^n$-grading via setting $\deg(x_i)=\deg(y_i)=e_i$, where $e_1,\ldots,e_n$ is a standard basis of~$\mathbb{Z}^n$. Let $G\subset {\mathbb C}^n$ be a subspace of codimension $r$ given by the equations \[a_{11}t_1+...+a_{1n}t_n=\dots=a_{r1}t_1+...+a_{rn}t_n=0.\] We consider Nagata's action of $G$ on $R$. As $x_i$ is invariant for every $i=1,...,n$, we can extend the action to the localization
\[
  R_{\mathbf x}=R[x_1^{-1},\dots,x_n^{-1}]=\Bbbk[x_1^{\pm 1},\dots,x_n^{\pm 1},\frac{y_1}{x_1},\dots,\frac{y_n}{x_n}].
\]

The grading on $R$ extends naturally to a grading on $R_{\mathbf x}$ with $\deg(x_i^{-1})=-e_i$. Now, $\lambda=(\lambda_1,\dots,\lambda_n)\in G$ acts on $R_{\mathbf x}$ by $x_i\mapsto x_i$ and $\frac{y_i}{x_i}\mapsto \frac{y_i}{x_i} + \lambda_i$. Let $y_i'= \frac{y_i}{x_i}$. Then $\lambda\in G$ acts on $\Bbbk[x_1^{\pm 1},...,x_n^{\pm 1}, y'_1,\dots,y'_n]$ by $x_i\mapsto x_i$ and $y'_i \mapsto y'_i+\lambda_i$. A direct computation shows that the invariant ring $R_{\mathbf x}^G$ is generated over $\Bbbk[x_1^{\pm 1},...,x_n^{\pm 1}]$ by the linear polynomials
\[
l'_i:=a_{i1}y'_1+ \cdots +a_{in}y'_n, \;\;\; 1\leq i \leq r.
\]
Let $x_0=\prod_{j=1}^n x_j$ and
\begin{equation}\label{eqn:li's}
  l_i=x_0\; \cdot l'_i=\big(x_0\big)\big(a_{i1}\frac{y_1}{x_1}+ \cdots +a_{in}\frac{y_n}{x_n}\big).
\end{equation}
We define the algebra $U:=\Bbbk[l_1,...,l_r]\subset R_{\mathbf x}\cap R^G$.
Let $V$ be the $\Bbbk$-vector space spanned by $l_1,\dots,l_r$. Then $U$ is a $\ZZ$-graded ring and $V$ is its degree one part.
We also let $V_i \subset V$ be the polynomials in $V$ that do not have $y_i\prod_{i\neq j}x_j$ as a monomial and $I_i \subset U$ the ideal generated by $V_i$.
Then we have the following:
\begin{proposition}
  The invariant algebra $R^G$ is the extended multi-Rees algebra
  \[
   U[x_1,\dots,x_n]+\sum_{d\in \ZZ^n}\big(I_1^{d_1}\cap \cdots \cap I_n^{d_n}\big)x_1^{-d_1} \cdots x_n^{-d_n} \subset U[x_1^{\pm 1},\ldots,x_n^{\pm 1}].
  \]
\end{proposition}
\begin{proof}
  A proof is found in \cite{Mukai} or in the book \cite{CoxRings}, section ~4.3.4.
  \qed
\end{proof}

\subsection{Cox Rings} The Cox ring of a smooth projective variety $X$ over the field $\Bbbk$, with finitely generated divisor class group $\operatorname{Cl}(X)$, is the ring
\[
\operatorname{Cox}(X)=\bigoplus_{(a_1,\dots,a_r)\in \ZZ^r}H^0(X,\cO_X(a_1D_1+\dots+a_rD_r)),
\]
where $D_1,\dots,D_r$ is a fixed basis of $\operatorname{Cl}(X)\simeq \ZZ^r$. This ring has the structure of a $\Bbbk$-algebra. When it is finitely generated the variety $X$ is called a Mori Dream Space. This is the case for smooth del Pezzo surfaces of degree $1\leq d\leq 9$, for which generators and relations among them are known.   

We let $A$ be an $r \times n$ matrix with entries in $\Bbbk$ such that $G$ is the kernel of~$A$. We denote by $a^{(i)}$ the $i$-th column vector of $A$ and assume that they are pairwise linearly independent. Denote by $X_G$ the del Pezzo surface resulting from the blow-up of $\PP^{r-1}$ at $n$ different points with homogeneous coordinates $a^{(i)}$. The del Pezzo surface $X_G$ is determined by $G$ only up to isomorphism: an isomorphism of $\PP^2$ as a linear map leaves the rowspace of $A$, and therefore also the kernel $G$, invariant and induces an isomorphism of the corresponding del Pezzo surfaces. The Picard group $\operatorname{Pic}(X_G)$ is isomorphic to $\ZZ^{n+1}$ and is generated by the proper transform of the hyperplane class $H$ and the classes of the exceptional divisors $E_i$ for $i=1,...,n$. Thus the Cox ring of $X_G$ is:
\[
\operatorname{Cox}(X_G)=\bigoplus_{(d_0,\dots,d_n)\in \ZZ^{n+1}}H^0(X_G,\mathcal{O}(d_0H+d_1E_1+\dots+d_nE_n)).
\]
Given a divisor class $D=d_0H+d_1E_1+\dots+d_nE_n$, the corresponding homogeneous part $\operatorname{Cox}(X_G)_D$ is the space
$H^0\big(X_G,\mathcal{O}(d_0H+\dots +d_nE_n) \big)$. If $d_0 \geq 0$ then $D$ is the class of the proper transform of a degree $d_0$ hypersurface that has multiplicity $-d_i$ in the point $a^{(i)}$. Thus we can identify $H^0(X_G,\mathcal{O}(d_0H+\dots +d_nE_n) )$ with the space of homogeneous polynomials of degree $d_0$ in $\Bbbk[z]=\Bbbk[z_1,\cdots,z_r]$ that have multiplicity at least $-d_i$ at $a^{(i)}$. Let $I_i'$ be the vanishing ideal in $\Bbbk[z]$ of the point $a^{(i)}$. Then the latter vector space is precisely
\begin{equation}\label{eqn:I'i}
\big(\big(I'_1\big)^{-d_1}\cap\cdots\cap\big(I'_n\big)^{-d_n}\big)_{d_0}
\end{equation}
where $(I'_i)^{-d_i}=\Bbbk[z]$ if $-d_i\leq 0$. If $d_0< 0$ then $H^0(X_G,\cO(D))=0$.

Let us consider the map
\[
\operatorname{Cox}(X_G)_D\simeq H^0(X_G,\cO(D))\longrightarrow R^G_d
\]
given by
\begin{equation*}
  g(z_1,\dots,z_r)\mapsto g(l_1,\dots,l_r)x_1^{d_1}\dots x_n^{d_n},
\end{equation*}
where $d=(d_0+d_1,\dots,d_0+d_n)$, and $l_i, 1\!\leq i \leq\! r$ are as in (\ref{eqn:li's}). Recall that $l_i=x_0l'_i$ where $l'_i\in R^G_{\mathbf x}$ are the invariants in $R_{\mathbf x}$ of degree $0\in \ZZ^n$. As $g$ is homogeneous of degree $d_0$, then $g(l_1,\dots,l_r)=x_0^{d_0}g(l'_1,\dots,l'_r)$ is an invariant of degree $(d_0,...,d_0)\in \ZZ^n$. Thus $g(l_1,\dots,l_r)x_1^{d_1}\dots x_n^{d_n}$ is indeed an element of $R^G$ of degree $(d_0+d_1,\dots,d_0+d_n)$.
Now we notice that
\[
R^G=U[x_1^{\pm 1},\dots,x_n^{\pm 1}]\cap R={\Bbbk}[l_1,\dots,l_r][x_1^{\pm 1},\dots,x_n^{\pm 1}]\cap R.
\]
Given $d\in \ZZ^n$, any homogeneous element $f\in R^G_d$ admits a presentation of the form
\[
f=\sum_{v \,\in \ZZ^n}h_{v}(l_1,\dots,l_r)x_1^{v_1}\dots x_n^{v_n}
\]
where the $h_v$ are homogeneous of degree $d-v$. On the other hand, the $l_i$ are homogeneous of degree $(1,\dots,1)\in \ZZ^n$ and therefore $d-v=(d_0,\dots,d_0)$ for some $d_0\geq 0$. Thus, $h_{v}(z_1,\dots,z_r)\in {\Bbbk}[z_1,...,z_r]_{d_0}$. Moreover, further calculations show that $h_{v}(l_1,\dots,l_r)x_1^{v_1}\dots x_n^{v_n}$ being a polynomial in $R$ implies that $h_{v}(l_1,\dots,l_r)\in (I_i)^{-v_i}$ and therefore $h_v(z_1,\dots,z_r)\in (I'_i)^{-v_i}$. Thus, given $d\in \ZZ^n$ fixed, we have an isomorphism
\begin{equation}\label{eqn:coxring-isomorphism}
  \bigoplus_
  {\makebox[0pt]{$\substack{D=(d_0,...,d_n)\in \ZZ^n,\\
      d=(d_0+d_1,\dots,d_0+d_n)}$}}\;\;\operatorname{Cox}(X_G)_D \longrightarrow(R^G)_{d},\;\;\;\;\;g(z)\mapsto g(l_1,\ldots,l_r)x_1^{d_1}\cdots x_n^{d_n}
  \end{equation}
where $d=(d_0+d_1,\dots,d_0+d_n)$. This and the previous proposition prove the following:
\begin{proposition}\label{prop:coxring-prop}
$\operatorname{Cox}(X_G)$ is isomorphic to $R^G$.
\end{proposition}
We should observe that the ideals $I'_i$ in (\ref{eqn:I'i}) do not change if we rescale the columns of $A$, yet the image of a polynomial $g$ under the isomorphism (\ref{eqn:coxring-isomorphism}) can be different.

\subsection{The Cox ring of a del Pezzo surface}\label{sect-cox-delPezzo}
In \cite{Batyrev-Popov} it was proven that the Cox ring of a del Pezzo surface of degree at least 2 is generated by the global sections over the exceptional curves. An exceptional curve is one with self-intersection $-1$. Such a curve has only one global section (up to scalar multiplication). We use this knowledge and the isomorphism of the previous part to compute a set of  generators for $R^G$.

\begin{example}\label{example_dP4}
  Before we move to the case of del Pezzo surfaces of degree~3, most important for us, let us say what the Cox ring of a del Pezzo surface of degree~4 looks like. This is a sketch of a solution to Problem~6 on Surfaces in~\cite{Sturmfels_Chapter}.

  We need to identify all exceptional curves on the blow-up of $\mathbb{P}^2$ in 5 points $P_1,\ldots,P_5$ in general position. First, there are 5 exceptional divisors of the blow-up, $E_1,\ldots, E_5$. Then one checks that strict transforms of lines through two points $P_i,P_j$ are exceptional curves. As divisors, they are linearly equivalent to $H-E_i-E_j$. Finally, there is one conic through all 5 chosen points, and its strict transform also is an exceptional curve, linearly equivalent to $2H - E_1 - E_2 - E_3 - E_4 - E_5$. Thus we have 16 generators of the Cox ring in total.

  Relations between them come, roughly speaking, from the possibility of decomposing a divisor class as sums of the ones given above in a few different ways. For instance, $2H-E_1-E_2-E_3-E_4$ can be written as:
\begin{align*}
  (H-E_1-E_2) + (H-E_3-E_4) &= (H-E_1-E_3) + (H-E_2-E_4) =\\ &= (H-E_1-E_4) + (H-E_2-E_3).
\end{align*}
This lead to relations of corresponding sections which generate the Cox ring. A good explanation of these computations (also for del Pezzo surfaces of smaller degree) can be found in the MSc thesis of J.C.~Ottem,~\cite{OttemMSc}. It is worth noting that different choices of points give different relations, but the Cox rings are isomorphic.
\end{example}

Now, let $G$ and $A$ be as before with $r=3$, $n=6$ and suppose that the points $a^{(i)}\in \PP^{r-1}$ are in general position, that is, no three of them lie on a line and no six on a conic. Then $X_G$ is a del Pezzo surface of degree $3$ and it has $27$ exceptional curves, determined in a very similar way as in Example~\ref{example_dP4}. These are the classes of:
\begin{itemize}
\item the exceptional divisors $E_i$, $1\leq i \leq 6$,
\item the proper transforms of lines which pass through pairs of the blown-up points $L_{ij}$, $1\leq i < j \leq 6$,  and
\item the proper transforms of conics through five of these points, $Q_i$ with $1\leq i \leq 6$. 
\end{itemize}
  The classes in $\operatorname{Pic}(X_G)$ of these curves are $E_i$, $H-E_i-E_j$ and $2H-\sum_{j\neq i}E_j$. This means that $\operatorname{Cox}(X)$ is generated by the images under \eqref{eqn:coxring-isomorphism} of the unique polynomials $g$ in $\Bbbk[z_1,\dots,z_r]$ having the prescribed multiplicity on the blown-up points. Now we compute these images explicitly. For simplicity we will denote $[6]=\{1,\dots,6\}$.

  Let us start with the exceptional divisors $E_i$. We have that the only monomials of degree $0$ in $\Bbbk[z]$ are the non-zero constants and they all belong to $(I'_i)^{-1}=\Bbbk[z]$. Thus, by \eqref{eqn:coxring-isomorphism} we get
  \begin{equation*}
    \big((I'_i)^{-1}\big)_0 ={\Bbbk[z]}_0\simeq\operatorname{Cox}(X_G)_{E_i}\simeq (R^G)_{e_i},
  \end{equation*}
where $e_i\in \ZZ^6$ is the $i$-th standard basis vector, and this isomorphism maps $1\mapsto 1\cdot x_i$. Thus the elements $\{x_i\,|1\leq i\leq 6\}$ are generators of $R^G$.   

For each class of the form $H-E_i-E_j$, there is a polynomial of degree one in $I'_i\cap I'_j$, namely, the equation of the unique line through the points $a^{(i)}$ and $a^{(j)}$. This is
\[
\big((a_{2j}a_{3i}-a_{2i}a_{3j})z_1+(a_{1i}a_{3j}-a_{1j}a_{3i})z_2+(a_{1j}a_{2i}-a_{1i}a_{2j})z_3\big).
\]
The image of this polynomial in $R^G$ is
\[
g(l_1,...,l_3)\cdot (x_1x_2)^{-1}=-\sum_{k\neq i,j}p_{ijk}y_k(\prod_{s\notin\{i,j,k\}}x_s),
\]
where the $p_{ijk}$ are the Pl\"ucker coordinates of $A$, and it has degree $\sum_{k\neq i,j}e_k\in \ZZ^6$.

Finally, for the class $2H-\sum_{j\neq m}E_j$ there is also a unique polynomial of degree $2$ in $\cap_{j\neq m}I'_j$: the defining polynomial of the unique conic through the five points different from $a^{(m)}$.
A direct computation shows that the image of this conic has the form
\[
  G_m=(x_m) \;\;\sum_{\makebox[0pt]{$\scriptstyle i<j, i,j\in [6]\setminus m$}}  \;\;p_{([6]\setminus{i,j,m})} \;\; y_i y_j \;\;\prod_{\makebox[0pt]{$\scriptstyle k\in [6]\setminus \{i,j,m\}$}}\;\; p_{ijk}x_k + (y_m)\cdot\sum _{i\in [6]} (u_i-v_i)y_i\prod_{k\neq i} x_k
\]
where $u_i-v_i$ is a binomial of degree 4 in the Pl\"ucker coordinates of $A$. This conic generator has degree $e_m+\sum_{i\in [6]}e_i$

It is worth noting that even when it is difficult to write the exact expression of the polynomials $G_m$, its computation is straightforward. Also, we observe that the generators of $R^G$ are determined up to scalar multiple by $G$ since the Pl\"ucker coordinates of the matrix $A$ are. Yet, as observed after Proposition \ref{prop:coxring-prop}, $R^G$ is not itself an invariant of the isomorphism class of $X_G$.
\subsection{Moneric and Khovanskii subspaces.} The preceding paragraphs show how a codimension~3 vector subspace $G \subset \Bbbk^n$, or a matrix containing its basis, gives a minimal generating set of the Cox ring of a del Pezzo surface of degree~$9-n$. Having covered this, we can finally introduce Khovanskii and moneric subspaces.

\begin{definition}\label{def_Ksubsp}
We say that a codimension~3 subspace $G \subset \Bbbk^n$ is \emph{Khovanskii} (resp. \emph{moneric}) if the corresponding minimal generating set of the Cox ring of a del Pezzo surface of degree $9-n$ is a Khovanskii (resp. moneric) basis of $R^G$.
\end{definition}

We would like to consider moneric and Khovanskii bases up to the following equivalence relation:

\begin{definition}\label{def_equivalence}
Codimension~3 subspaces $G, G' \subset \Bbbk^n$ will be called \emph{equivalent} if the corresponding initial algebras of the Cox ring of a del Pezzo surface are equal.
\end{definition}

Note that if $G$ and $G'$ determine the same initial terms of the minimal generating set of corresponding Cox rings then they are equivalent.

\section{Khovanskii Basis and Degeneration of the Cox Ring}\label{section_degeneration}
Degeneration of  varieties is a powerful tool in algebraic geometry, used on many different occasions. The idea behind it is to introduce a notion of a ``limit"  of a family of algebraic varieties. However, since the Zariski topology on an algebraic variety is not well behaved in this sense (it is for example almost never Hausdorff), it turns out that a better replacement for an arbitrary family of varieties is the notion of a flat family. This notion has the desirable feature that limit points exist and are unique if we parametrize over a one dimensional variety. It also ensures that the points in the family, including the limit point have the same Hilbert function, and thus share many invariants such as e.g. the degree and the genus.
Degenerations thus motivate the following approach:
to compute properties of a given variety $X$ first degenerate the variety to a  more accessible variety $X^\prime$ and then do the computations on this variety.
This idea can be realized in the notion of a Khovanskii basis.

\subsection{Toric degenerations}
The following definition makes precise what we mean by a degeneration of a variety.
\begin{definition}
Let $(\Bbbk^{\circ}, \mathfrak{m})$ be a discrete valuation ring and $X$ be a variety over $\Bbbk=\operatorname{Quot}(\Bbbk^{\circ})$. A \emph{degeneration} of the variety $X$ is a flat family $\tilde{X}\to \operatorname{Spec} (\Bbbk^{\circ})$ such that $\tilde{X}\times_{\Bbbk^{\circ}} \operatorname{Spec} (\Bbbk)\cong X$. It is called a \emph{toric} degeneration if the special fiber $\tilde{X}\times_{\Bbbk^{\circ}} \operatorname{Spec} (\Bbbk^{\circ}/\mathfrak{m})$ is a toric variety.
\end{definition}
In this section we provide a method for degenerating a variety with respect to all possible embeddings at once. The idea is to degenerate the Cox ring of the given variety which contains information about all possible embeddings of the variety.
In order to talk about degenerations of a projective variety with respect to a specific embedding, we need to take the choice of a very ample line bundle into account.
\begin{definition}
Let $(\Bbbk^{\circ}, \mathfrak{m})$ be a discrete valuation ring and $X$ be a projective variety over $\Bbbk$, together with a very ample line bundle $L$. A family $\tilde{X}\to \operatorname{Spec}(\Bbbk^{\circ})$ together with a line bundle $\tilde{L}$ is called a \emph{toric degeneration of} $X$ \emph{with respect to the embedding given by}~$L$ if it is a toric degeneration, $\tilde{L}$ is flat over $\operatorname{Spec}(\Bbbk^{\circ})$, we have $L_{|\tilde{X}\times \operatorname{Spec}(\Bbbk)}\cong L$ and the line bundle $L_{|\tilde{X}\times \operatorname{Spec}(\Bbbk^{\circ}/\mathfrak{m})}$ is ample.
\end{definition}
Note that in the above definition we did not assume that $L_{|\tilde{X}\times \operatorname{Spec}(\Bbbk^{\circ}/\mathfrak{m})}$ is very ample. However, if we consider the Veronese embedding of the embedded variety $X$, we may assume that $X$ as well as the special fiber are embedded in the same $\PP^N$. More concretely by replacing $L$ with a high enough multiple $L^{\otimes k}$, we can make sure that $L_{|\tilde{X}\times \operatorname{Spec}(\Bbbk^{\circ}/\mathfrak{m})}$ is also very ample.

\subsection{Degenerations of del Pezzo surfaces via the Cox Ring}
Let $\Bbbk=F(t)$ for a field $F$ of characteristic $0$. We often assume $F=\mathbb{Q}$.
As in section~\ref{section_cox}, given $n\in \{1,\dots,8\}$ we can associate to a matrix $A\in Mat_{\Bbbk}(3,n)$ which has maximal rank, with kernel $G$, the variety $X_G$. The variety $X_G$ is the blow-up of $\PP^{2}$ at the points represented by $A$. Proposition \ref{prop:coxring-prop} gives us the following identity 
\[\operatorname{Cox}(X_G)\simeq \Bbbk[x_1,\dots ,x_n, y_1,\dots,y_n]^G=:R^G.\]
By varying the variable $t$ we can interpret $X_G$ as a family of del Pezzo surfaces over $F$ and $\operatorname{Cox}(X_G)$ as the corresponding family of Cox rings.
Note however that the only property we are using in this section about the variety $X_G$ is that its Cox ring $R_G$ is a subalgebra of a polynomial ring $\Bbbk[x_1,\dots,x_r]$.

Let $\Bbbk^{\circ}$ be the corresponding valuation ring of $\Bbbk$, i.e. the set of all elements having nonnegative valuation.

\begin{theorem}
Let $R\subset \Bbbk[x_1,\dots,x_n]$ be an algebra.
A finite Khovanskii basis $\mathcal{F}$ of $R$ induces a toric degeneration of $\operatorname{Spec}(R)$.
\end{theorem}
\begin{proof}
Let $\mathcal{F}$ be a finite Khovanskii basis.
Let us denote for $f\in \mathcal{F}$ by $\operatorname{trop}(f)(\mathbf{1})$ the minimum of the valuation of the coefficients of $f$.
Consider the $\Bbbk^{\circ}$-algebra
 \[R^G_{\Bbbk^{\circ}}:= \{ t^{\operatorname{trop}(f)(\mathbf{1})}f \ |\ f\in R^G\}\subset \Bbbk^{\circ}[x_1,\dots,x_n].
\] 
  We claim that $\operatorname{Spec}(R^G_{\Bbbk^{\circ}}) \to \operatorname{Spec} (\Bbbk^{\circ})$ is a toric degeneration of $\operatorname{Spec}(R^G)$.

 It is a flat morphism since $R^G_{\Bbbk^{\circ}}$ is a torsion free module over the discrete valuation ring $\Bbbk^{\circ}$. 
Now, the general fiber is given by 
\[
\operatorname{Spec} (R^G_{\Bbbk^{\circ}}\otimes_{\Bbbk^{\circ}} \Bbbk)\cong\operatorname{Spec}(R^G),
\]
 and the special fiber is 
\[ 
 \operatorname{Spec}(R^G_{\Bbbk^{\circ}}\otimes_{\Bbbk^{\circ}} \Bbbk^{\circ}/(t))\cong\operatorname{Spec}(\operatorname{in}(R^G)).
\]
The last thing to prove is that the algebra $\operatorname{in}(R^G_{\Bbbk^{\circ}})$ is an affine semigroup algebra. But this follows easily from the fact that it is a finitely generated algebra generated by monomials.
\qed
\end{proof}
As a consequence of the above theorem we conclude that a finite Khovanskii basis $\mathcal{F}$ of $R^G$ induces a toric degeneration of $\operatorname{Spec}(R^G)$.
Now we want to show how this toric degeneration gives a toric degeneration of $X_G$ with respect to any embedding. For this purpose the following lemma is helpful.
\begin{lemma}
Let $\mathcal{F}$ be a finite Khovanskii basis of $R^G$.
Let $L$ be a very ample line bundle on $X_G$ and $T:=\bigoplus T_q:= \bigoplus_{q\in {\NN}_0} H^0(X,L^{\otimes q})\subset R^G$ be its graded section ring. Then $\operatorname{in}(T)$ is finitely generated. 
\end{lemma}
\begin{proof}
Let $f_1,\dots, f_w\in R^G$ be homogenous elements which form a Khovanskii basis of $R^G$. For each $\beta \in \NN_0^w$ consider the set of all polynomials $f_{\beta}:=\prod_{i=1}^w f_i^{\beta_i} $ such that there is a non-negative integer $p\in \NN$ for which we have 
\begin{equation} \label{khovanskiiproperty} \sum_{i=1}^{w} \beta_i \cdot \deg(f_i)= p\cdot \deg(L). 
\end{equation}
By a slight abuse of notation, we use $\deg$ for the function which assigns to a section as well as to a divisor the corresponding integer vector under the isomorphism $\operatorname{Pic}(X_G)\cong \mathbb{Z}^{n+1}$.
Our first claim is that this set forms a (possibly non-finite) Khovanskii basis of~$T$.
Indeed, let $f\in T_q$ be a homogeneous element. Using the assumption that the $f_i$'s form a Khovanskii basis for $R^G$, we deduce that there are finitely many $\alpha_j \in \NN_0^r$, and $c_j\in k$ which satisfy \[ \operatorname{in}(f)= \sum_{j} c_j\cdot \operatorname{in} (f_{\alpha_j}). \]
Since $f$ was homogeneous, the degrees of all  the $f_{\alpha_j}$ match the degree of $f$, and we deduce that all the $f_{\alpha_j}$ fulfill the above prescribed property of equation \eqref{khovanskiiproperty}.

Next, we want to prove that finitely many $f_{\beta}$ suffice to form a Khovanskii basis. The question can be reformulated into the question of the finite generation of the following semigroup:
\[ S:=\{ (\beta_1,\dots, \beta_w,k)\in \NN_0^{r}\times \NN \ | \ \sum_{i=1}^w \beta_i\cdot \deg(f_i)=k\cdot \deg (L)\}. \] Consider the cone $C(S)$ generated by
$S$ in $ \mathbb{R}^{w+1}$. Then  $C(S)\cap \mathbb{Z}^{w+1}=S$, hence  by Gordan's Lemma $S$ is finitely generated.
\qed
\end{proof}

\begin{theorem}\label{Degeneration_Second_Theorem}
A finite Khovanskii basis of $R^G=\operatorname{Cox(X_G)}$ induces a toric degeneration of $X_G$ with respect to all possible embeddings.
\end{theorem}
\begin{proof}
Let $L$ be a very ample line bundle on $X$ and let $T:=\bigoplus_{q\in \NN_{0}} H^0(X,L^{\otimes q})$ be its graded section algebra.
Define the algebra
\[
 T_{\Bbbk^\circ}=\{t^{-\operatorname{trop}(f)(\mathbf{1})}f \ | \ f\in T\}\subset \Bbbk^{\circ}[x_1,\dots,x_n,y_1,\dots, y_n],  \]
where again  $\operatorname{trop}(f)(\mathbf{1})$ denotes the minimum of the valuation of the coefficients of $f$.
 The $\NN_0$-grading on $T$ defines a natural grading on $T_{\Bbbk^\circ}$.
As $L$ is very ample the section ring $T$ is finitely generated. Hence, the same follows for the graded algebra $T_{\Bbbk^{\circ}}$.
The flatness of $T_{\Bbbk^{\circ}}$ can easily be derived from the torsion freeness over the discrete valuation ring $\Bbbk^{\circ}$.
Thus we get an induced flat morphism
\[ \operatorname{Proj} ( \bigoplus_{q\in\NN_0} (T_{\Bbbk^{\circ}})_q ) \to \operatorname{Spec} (\Bbbk^{\circ})
\]
and an induced  line bundle $\tilde{L}= \mathcal{O}_{T_{\Bbbk^{\circ}}}(1)$.

For the computations of the fibers, we use the following two identities on the graded pieces:
\begin{align*} &(T_{\Bbbk^{\circ}})_q\otimes_{\Bbbk^{\circ}} k\cong\operatorname{in} (T_q),\\
&(T_{\Bbbk^{\circ}})_q\otimes_{\Bbbk^{\circ}} \Bbbk\cong T_q,\end{align*}
where $\operatorname{in}(T_q)$ is the $\Bbbk^\circ/\mathfrak{m}$-vector space generated by $\operatorname{in}(f)$ for all $f\in T_q$.

Therefore we get the following:
\begin{align*}
&\operatorname{Proj}(\bigoplus_{q\in\mathbb{N}_0} (T_{\Bbbk^{\circ}})_q)\times \operatorname{Spec} (\Bbbk)= \operatorname{Proj}(\bigoplus_{q\in\mathbb{N}_0} (T_{\Bbbk^{\circ}})_q\otimes \Bbbk)=\operatorname{Proj}(\bigoplus_{q\in\mathbb{N}_0} T_q)\cong X,\\
&\operatorname{Proj}(\bigoplus_{q\in\mathbb{N}_0} (T_{\Bbbk^{\circ}})_q)\times \operatorname{Spec}( k)=\operatorname{Proj}(\bigoplus_{q\in\mathbb{N}_0} (T_{\Bbbk^{\circ}})_q\otimes k)= \operatorname{Proj}(\bigoplus_{q\in\mathbb{N}_0} \operatorname{in}(T_q))=:X_T.
\end{align*}
The previous lemma implies that the graded algebra $\operatorname{in}(T)=\bigoplus_{q\in\mathbb{N}_0} \operatorname{in}(T_q)$ is finitely generated by monomials, and can be seen as the semigroup algebra of 
\[
S:=\bigoplus_{q\in\NN_0} S_q:=\bigoplus_{ q\in\NN_0}\{ (\beta_1,\dots, \beta_w,q)\in \NN_0^w\times \{q\}\ | \ \sum \beta_i\cdot \deg(f_i)=k\cdot \deg (L)\}. \]
This shows that $X_T=\operatorname{Proj}(\bigoplus_{q\in\mathbb{N}_0} \operatorname{in}(T_q))$ is a toric variety and $\tilde{L}_{|\operatorname{Proj}(X_T)}= \mathcal{O}_{X_T}(1)$ is the induced ample line bundle.
\qed
\end{proof}

\section{Motivation: Hilbert functions of del Pezzo surfaces and Ehrhart-type formulas}\label{section_counting}
\label{sec:motivation}
\noindent The original motivation of the paper \cite{SturmfelsEtAl} is to give an interpretation of the Hilbert function of the Cox-Nagata ring $R^G$ as a counting function of the numbers of lattice points in slices of some explicit rational convex polyhedral cone. If we focus on a specific embedding of the variety $X_G$ into a projective space, then this interpretation induces a realization of the Hilbert function of $X_G$ with respect to the embedding as the Ehrhart function of an explicit rational convex polytope. 

Such an Ehrhart-type formula has appeared in many areas of mathematics: Berenstein-Zelevinsky's description of tensor product multiplicities for representations \cite{BerensteinZelevinsky}, Holtz-Ron's work on zonotopal algebras \cite{HoltzRon}, the theory of Newton-Okounkov bodies \cite{KavehKhovanskii, LazarsfeldMustata}, and so forth.  Having an Ehrhart-type formula for a mathematical object enables us  to relate it with many areas of mathematics through convex geometry. One more important point is that an Ehrhart-type formula is easy to compute since a polytope is bounded and given by a finite number of inequalities.

The theory of Khovanskii bases gives a systematic way to construct an Ehrhart-type formula for the Hilbert function of a graded ring under some assumptions. We explain this construction following \cite{SturmfelsEtAl}. Let $\Bbbk$ be the rational function field $\mathbb{Q}(t)$, $\Bbbk[x_1, \ldots, x_m]$ the polynomial ring over $\Bbbk$ in $m$ variables, and ${\bf x}^{\bf a} := x_1 ^{a_1} \cdots x_m ^{a_m}$ for ${\bf a} = (a_1, \ldots, a_m) \in \mathbb{Z}_{\ge 0} ^m$. Note that the residue field $k$ is identical to the field $\mathbb{Q}$ of rational numbers. Fix ${\bf d}_1, \ldots, {\bf d}_m \in \mathbb{Z}^n$, and define a $\mathbb{Z}^n$-graded $\Bbbk$-algebra structure on $\Bbbk[x_1, \ldots, x_m]$ by $\deg(x_i) := {\bf d}_i$ for $1 \le i \le m$. We assume that the homogeneous parts $\Bbbk[x_1, \ldots, x_m]_{\bf d}$, ${\bf d} \in \mathbb{Z}^n$, are finite-dimensional. Let $U$ be a $\mathbb{Z}^n$-graded $\Bbbk$-subalgebra of $\Bbbk[x_1, \ldots, x_m]$ with a finite Khovanskii basis $\mathcal{F} \subset U$. The $\mathbb{Z}^n$-grading of $U$ induces a $\mathbb{Z}^n$-graded $k$-algebra structure on ${\rm in}(U)$. The {\it Hilbert function} of $U$ is a map $\psi: \mathbb{Z}^n \rightarrow \mathbb{Z}_{\ge 0}$ given by \begin{equation*}\psi({\bf d}) := \dim_\Bbbk(U_{\bf d})\end{equation*} for ${\bf d} \in \mathbb{Z}^n$. Let $\mathbb{Z}_{\ge 0}({\rm in}(\mathcal{F}))$ (resp. $\mathbb{Z}({\rm in}(\mathcal{F}))$) be the subsemigroup (resp. the subgroup) of $\mathbb{Z}^m$ generated by $\{{\bf a} \in \mathbb{Z}_{\ge 0} ^m \mid {\bf x}^{\bf a} \in {\rm in}(\mathcal{F})\}$ and the zero vector. Denote by $\Gamma \subset \mathbb{R}^m$ the smallest real closed convex cone containing $\mathbb{Z}_{\ge 0}({\rm in}(\mathcal{F}))$. The $\mathbb{Z}^n$-graded $\Bbbk$-algebra structure on $\Bbbk[x_1, \ldots, x_m]$ induces a $\mathbb{Z}^n$-grading of the semigroup $\Gamma \cap \mathbb{Z}({\rm in}(\mathcal{F}))$. We observe that ${\rm in}(U)$ is identical to the semigroup algebra of $\mathbb{Z}_{\ge 0}({\rm in}(\mathcal{F}))$, which is regarded as a $\mathbb{Z}^n$-graded $k$-subalgebra of the semigroup algebra of $\Gamma \cap \mathbb{Z}({\rm in}(\mathcal{F}))$.

\begin{proposition}\label{Motivation_Proposition}
If the initial algebra ${\rm in}(U)$ is normal, then the value $\psi({\bf d})$ for ${\bf d} \in \mathbb{Z}^n$ equals the cardinality of 
\begin{equation*}
\{{\bf a} \in \Gamma \cap \mathbb{Z}({\rm in}(\mathcal{F})) \mid \deg({\bf a}) = {\bf d}\}.
\end{equation*}
\end{proposition}

\begin{proof}
Fix ${\bf d} \in \mathbb{Z}^n$ such that $U_{\bf d} \neq \{0\}$, and take a $\Bbbk$-basis $\{f_1, \ldots, f_r\}$ of $U_{\bf d}$. If the initial forms ${\rm in}(f_1), \ldots, {\rm in}(f_r)$ are linearly dependent, then the definition of initial forms implies that there exist $c_1, \ldots, c_r \in \Bbbk$ such that ${\rm in}(c_1 f_1 + \cdots + c_r f_r)$ does not belong to the $k$-linear space spanned by ${\rm in}(f_1), \ldots, {\rm in}(f_r)$. Then by replacing $f_i$ for some $1 \le i \le r$ with $c_1 f_1 + \cdots + c_r f_r$, we can increase the dimension of the $k$-linear space spanned by ${\rm in}(f_1), \ldots, {\rm in}(f_r)$. Repeating this procedure, we obtain a $\Bbbk$-basis $\{\tilde{f}_1, \ldots, \tilde{f}_r\}$ of $U_{\bf d}$ such that the initial forms ${\rm in}(\tilde{f}_1), \ldots, {\rm in}(\tilde{f}_r)$ are linearly independent. 

Then it follows that these form a $k$-basis of ${\rm in}(U)_{\bf d}$. In particular, the $\Bbbk$-algebra $U$ and its initial algebra ${\rm in}(U)$ share the same Hilbert function. Since ${\rm in}(U)$ is identical to the semigroup algebra of $\mathbb{Z}_{\ge 0}({\rm in}(\mathcal{F}))$, the group $\mathbb{Z}({\rm in}(\mathcal{F}))$ is regarded as a subset of the field of fractions of ${\rm in}(U)$. Hence the normality assumption on ${\rm in}(U)$ implies that the semigroup $\mathbb{Z}_{\ge 0}({\rm in}(\mathcal{F}))$ is saturated in $\mathbb{Z}({\rm in}(\mathcal{F}))$, and hence that $\mathbb{Z}_{\ge 0}({\rm in}(\mathcal{F})) = \Gamma \cap \mathbb{Z}({\rm in}(\mathcal{F}))$. In particular, the initial algebra ${\rm in}(U)$ is identical to the semigroup algebra of $\Gamma \cap \mathbb{Z}({\rm in}(\mathcal{F}))$. This proves the proposition. \qed
\end{proof}

\begin{remark}\normalfont
Our proof of Theorem \ref{Degeneration_Second_Theorem} in Section \ref{section_degeneration} also uses initial forms. In the case $U = R^G$, the description of $\psi$ in Proposition \ref{Motivation_Proposition} reflects the toric degeneration of $X_G$ constructed in the theorem. Assume that there exists a finite Khovanskii basis of $R^G$. Let us fix a very ample line bundle $L$ on $X_G$, and take a multi-degree ${\bf d}$ such that $(R^G)_{\bf d} = H^0(X_G, L)$. Then the Hilbert polynomial of $X_G$ with respect to the corresponding embedding is identical to the polynomial in $l$ given by $\psi(l{\bf d})$ for $l \gg 0$. In addition, by \cite[Chapter I\hspace{-.1em}I\hspace{-.1em}I, Theorem 9.9]{Hartshorne}, the Hilbert polynomial of $X_G$ is identical to that of the resulting toric variety from the toric degeneration. From these and our proof of Theorem \ref{Degeneration_Second_Theorem}, we obtain an Ehrhart-type description of $\psi(l{\bf d})$ for $l \gg 0$, which is identical to the formula in Proposition \ref{Motivation_Proposition}.
\end{remark}

Normality (or saturatedness) is a key to an Ehrhart-type formula in general. In the case of ${\rm in}(U)$, the theory of Gr\"{o}bner bases can be applied to prove the normality as follows. Since $\mathcal{F}$ is moneric, we deduce that ${\rm in}(U)$ is generated by a finite number of monomials, and hence that the ideal $I$ of relations is spanned by a set of binomials (\cite[Lemma 4.1]{Sturmfels_Book}). Then we obtain a useful sufficient condition for the normality of ${\rm in}(U)$ in terms of a Gr\"{o}bner basis of $I$ (see \cite[Proposition 13.15]{Sturmfels_Book}).

\begin{example}[elementary symmetric function]\label{Motivation_first_example}\normalfont
Following \cite[Example 3.2]{SturmfelsEtAl}, set \begin{equation*}e_l (t, x_1, \ldots, x_m) := \sum_{1 \le j_1 < \cdots < j_l \le m} t^{(j_1 - 1) + (j_2 - 2) + \cdots + (j_l - l)} x_{j_1} \cdots x_{j_l} \in \Bbbk[x_1,\dots, x_m] \end{equation*} for $1 \le l \le m$, and $\mathcal{F} := \{e_l (t, x_1, \ldots, x_m) \mid 1 \le l \le m\} \subset \Bbbk[x_1, \ldots, x_m]$.

We obtain the elementary symmetric functions by specializing at $t = 1$. Let $U$ be the $\Bbbk$-subalgebra of $\Bbbk[x_1, \ldots, x_m]$ generated by $\mathcal{F}$. It is easily checked that the initial algebra ${\rm in}(U)$ is identical to the $k$-subalgebra of $k[x_1, \ldots, x_m]$ generated by $\{x_1, x_1 x_2, \ldots, x_1 x_2 \cdots x_m\}$, and hence that $\mathcal{F}$ is a Khovanskii basis. In addition, the initial algebra ${\rm in}(U)$ is normal since $x_1, x_1 x_2, \ldots, x_1 x_2 \cdots x_m$ are algebraically independent. 

Since ${\rm in}(\mathcal{F}) = \{x_1, x_1 x_2, \ldots, x_1 x_2 \cdots x_m\}$, we have $\mathbb{Z}({\rm in}(\mathcal{F})) = \mathbb{Z}^m$ and
\begin{equation}\label{Motivation_equation}
\Gamma \cap \mathbb{Z}^m = \{(a_1, \ldots, a_m) \in \mathbb{Z}^m \mid a_1 \ge a_2 \ge \cdots \ge a_m \ge 0\}.
\end{equation}
Regard $U$ as a $\mathbb{Z}_{\ge 0}$-graded $\Bbbk$-algebra by the total degree in variables $x_1, \ldots, x_m$. Let $\psi: \mathbb{Z}_{\ge 0} \rightarrow \mathbb{Z}_{\ge 0}$ denote the Hilbert function. We deduce by Proposition \ref{Motivation_Proposition} and equation (\ref{Motivation_equation}) that the value $\psi(r)$ for $r \in \mathbb{Z}_{\ge 0}$ equals the cardinality of \begin{equation*}\{(a_1, \ldots, a_m) \in \mathbb{Z}^m \mid a_1 \ge a_2 \ge \cdots \ge a_m \ge 0,\ a_1 + \cdots + a_m = r\};\end{equation*} this is identical to the set of partitions of $r$ with at most $m$ parts.
\end{example}

Let us come back to our situation of interest.

\begin{example}[{\cite[Theorem 3.5 and Proposition 3.6]{SturmfelsEtAl}}]\normalfont\label{Motivation_second_example}
Let $G \subset \Bbbk^n$ be a generic subspace of dimension $1$, and $(\alpha_1, \ldots, \alpha_n) \in \Bbbk^n$ a nonzero element of $G$. We consider a $(2 \times n)$-matrix \begin{equation*}\begin{pmatrix}
\alpha_1 x_1 & \alpha_2 x_2 & \cdots & \alpha_n x_n\\
y_1 & y_2 & \cdots & y_n
\end{pmatrix},\end{equation*} and, for $1 \le i < j \le n$, denote by $p_{ij}$ the $(2 \times 2)$-minor of this matrix with column indices $i, j$, that is, $p_{ij} = \alpha_i x_i y_j - \alpha_j x_j y_i$. If we regard $\alpha_i x_i$ as an indeterminate, then the $\Bbbk$-subalgebra of $R^G$ generated by $\{p_{ij} \mid 1 \le i < j \le n\}$ is identical to the homogeneous coordinate ring of the Grassmann variety of lines in the $(n-1)$-dimensional projective space over $\Bbbk$ with respect to the usual Pl\"{u}cker embedding. In particular, the minors $p_{ij}$, $1 \le i < j \le n$, satisfy the Pl\"{u}cker relation: \begin{equation*}p_{il} p_{jm} - p_{im} p_{jl} = p_{ij} p_{lm}\end{equation*} for $1 \le i < j < l < m \le n$. This is a key to the fact that $\mathcal{F} := \{p_{ij} \mid 1 \le i < j \le n\} \cup \{x_i \mid i = 1, \ldots, n\}$ is a Khovanskii basis of $R^G$. The normality of the initial algebra ${\rm in}(R^G)$ follows from the criterion explained before Example~\ref{Motivation_first_example}. Thus we can apply Proposition \ref{Motivation_Proposition} to obtain an Ehrhart-type formula for the Hilbert function $\psi$ of $R^G$. We may assume without loss of generality that ${\rm val}(\alpha_1) < \cdots < {\rm val}(\alpha_n)$. Then since ${\rm in}(p_{ij}) \in k^\ast x_i y_j$, we deduce that the value of $\psi$ at $(r, u_1, \ldots, u_n) \in \mathbb{Z}^{n+1}$ equals the number of $(a_{i, j})_{i = 1, 2,\ j = 1, \ldots, n} \in \mathbb{Z}_{\ge 0} ^{2n}$ satisfying the following conditions:

\begin{align*}
&a_{2, 1} = 0,\ a_{2, 2} + \cdots + a_{2, l+1} \le a_{1, 1} + \cdots + a_{1, l},\ 1 \le l \le n-1,\\
&a_{1, l} + a_{2, l} = u_l,\ 1 \le l \le n,\ a_{2, 1} + \cdots + a_{2, n} = r.
\end{align*}
\end{example}

By \cite[Theorem 6.1]{SturmfelsEtAl}, for $4 \le n \le 8$, there exists a generic $\Bbbk$-subspace $G \subset \Bbbk^n$ of codimension $3$ such that $R^G$ has a finite Khovanskii basis $\mathcal{F}$ and ${\rm in}(R^G)$ is normal. Hence Proposition \ref{Motivation_Proposition} produces an Ehrhart-type formula for the Hilbert function $\psi$ of the $\mathbb{Z}^{n+1}$-graded algebra $R^G$. The case of degree $5$ del Pezzo surfaces is included in Example \ref{Motivation_second_example}. In the case of del Pezzo surfaces of degree $3$ and $4$, Sturmfels and Xu gave a system of explicit linear inequalities defining the corresponding rational convex polytope for a specific subspace $G$ (\cite[Example 1.3 and Corollary 5.2]{SturmfelsEtAl}). 

Since $G$ is generic, the function $\psi$ is independent of the choice of $G$. Hence we obtain a system of Ehrhart-type formulas for the same function $\psi$. If $G$ is a different generic subspace, then the induced Ehrhart-type formula may be different, that is, the corresponding rational convex polytope may not be unimodular equivalent. 
In order to compute $\psi$ rapidly, we want to determine a generic subspace $G$ such that the number of linear inequalities defining the corresponding rational convex polytope is as small as possible. 

In case of del Pezzo surfaces of degree $4$, Sturmfels and Xu proved that the optimal number of linear inequalities is $12$. Their proof relies on giving the complete classification of the subspaces $G$ which produce Khovanskii bases (\cite[Theorem 4.1]{SturmfelsEtAl}). In addition, they conjectured that in the case of degree $3$ the number $21$ of linear inequalities in \cite[Corollary 5.2]{SturmfelsEtAl} is minimal. One motivation of this research is to generalize their argument for degree $4$ del Pezzo surfaces to the case of degree $3$, and to prove the conjecture by giving a complete characterization of all subspaces $G$ for which $R^G$ has a finite Khovanskii basis.

\section{Tropicalization}\label{section_trop}

Tropicalization is a procedure that associates to a very affine variety $X$ (i.e. a closed subvariety of an algebraic torus) a rational polyhedral complex $\operatorname{Trop}(X)$ in $\mathbb{R}^N$. Of the many ways characterizing $\operatorname{Trop}(X)$, there are two descriptions that will be useful for our purposes.  In terms of initial degenerations, $\operatorname{Trop}(X)$ is the set of all $w \in \mathbb{R}^N$ such that $\operatorname{in}_w X$ is nonempty (note that the ideal of $\operatorname{in}_w X$ coincides with the initial ideal as defined in \cite{BossingerEtAl} in the case where the valuation on $\mathbb{K}$ is trivial; for the definition of $\operatorname{in}_w X$ in general, see \cite[Section 5]{Gubler}). This allows us to compute $\operatorname{Trop}(X)$ using computer algebra software such as \emph{gfan}~\cite{gfan}. When $X$ is defined over an algebraically closed field with a nontrivial valuation, $\operatorname{Trop}(X)$ is the closure of the set of coordinatewise valuations. As this is the description we use for our classification problem, we will provide a more precise formulation of this characterization. 

Let $\mathbb{K}$ be a field with a (possibly trivial) valuation $\operatorname{val}: \mathbb{K}^* \to \mathbb{R}$, and $X$ a closed subvariety of the algebraic torus $\mathbb{G}_m^N(\mathbb{K})$. We define the Beri-Groves set $\mathcal{A}(X)$ of $X$ to be 
\begin{equation*}\mathcal{A}(X) = \{(\operatorname{val}(x_1), \ldots, \operatorname{val}(x_N)) \in \mathbb{R}^N \; | \; (x_1,\ldots, x_N)\in X\}. \end{equation*}
Now, suppose $\mathbb{L}$ is an algebraically closed field extension of $\mathbb{K}$ with a nontrivial valuation extending the the valuation on $\mathbb{K}$. By abuse of notation, we will also call this valuation $\operatorname{val}:\mathbb{L}^* \to \mathbb{R}$. Let $X_{\mathbb{L}}$ denote the extension of $X$ to a closed subvariety of $\mathbb{G}_m^N(\mathbb{L})$. Tropicalization is unchanged under field extension, i.e. $\operatorname{Trop}(X_{\mathbb{L}}) = \operatorname{Trop}(X)$, see~\cite[Thm 3.2.4]{MaclaganSturmfels}.  By the Fundamental Theorem of Tropical Geometry~\cite[Thm 3.2.3]{MaclaganSturmfels}, the closure of $\mathcal{A}(X_{\mathbb{L}})$ in $\mathbb{R}^N$ is $\operatorname{Trop}(X_{\mathbb{L}})$. Moreover, if the valuation on $\mathbb{K}$ is trivial, then $\operatorname{Trop}(X)$ is a rational polyhedral \textit{fan} in $\mathbb{R}^N$.

Now let us specialize to the case of the Tropical Grassmannian. We let $\mathbb{K} = \mathbb{Q}$ and $\mathbb{L}$ will denote Puiseux series over $\mathbb{C}$.  The Grassmannian $\Gr(d,\mathbb{Q}^n)$ can be viewed as a subvariety of $\mathbb{P}^{N-1}(\mathbb{Q})$ via its Pl\"ucker embedding, where $N = {n \choose d}$. Let $\Gr_0(d,n)$ be the intersection of the affine cone of $\Gr(d,\mathbb{Q}^n)$ with the dense torus (the locus where all Pl\"ucker coordinates are nonzero). This gives us a closed subvariety of $\mathbb{G}_m^N(\mathbb{Q})$, so we may form the tropicalization $\operatorname{Trop}(\Gr_0(d,\mathbb{Q}^n))$. Let us abbreviate this by $\operatorname{TGr}(d,\mathbb{Q}^n)$. This is a rational polyhedral fan in $\mathbb{R}^N$. We index the coordinates of $\mathbb{R}^N$ by the $d$-tuples of the numbers 1 through $n$.  In \cite{SpeyerEtAl}, Speyer and Sturmfels give a combinatorial description of $\operatorname{TGr}(2,\mathbb{Q}^n)$ in terms of the space of phylogenetic trees on $n$ leaves (up to sign). In particular, they show that $d = (d_{ij})$ is a point in $\operatorname{TGr}(2,\mathbb{Q}^n)$ if and only if for each 4-tuple $1\leq i<j<k<l\leq n$, the maximum of 
\begin{equation*}d_{ij}+d_{kl}, \;\; d_{ik}+d_{jl},\;\; d_{il}+d_{jl} \end{equation*} 
is attained at least twice. 

In the classification of Khovanskii subspaces $G$ of $\Bbbk^5$, $G$ can be viewed as a $\Bbbk$-valued point of $\Gr_0(2,\mathbb{Q}^5)$, where $\Bbbk=\mathbb{Q}(t)$. If $(p_{ij})$ are the Pl\"ucker coordinates of $G$, then the valuations of the Pl\"ucker coordinates $d_{ij} = -\operatorname{val}(p_{ij})$ are integers. This means that the Beri-Groves set of the $\Bbbk$-valued points of $\Gr_0(2,\mathbb{Q}^5)$ is the set of integer points in $\operatorname{TGr}(2,\mathbb{Q}^5)$.

The Naruki fan is a fan structure on the tropicalization of the moduli space of marked del Pezzo surfaces of degree 3. Let $Y^6$ be the moduli space of degree 3 marked del Pezzo surfaces. We can express $Y^6$ as an open subvariety of the space of configurations of 6 labeled points in $\mathbb{P}^2$ in linear general position; call this space $X^6$. By the Gelfand-MacPherson correspondence, we can recover $X^6$ from the space of $3\times 6$ matrices by taking appropriate quotients. Recall that the Grassmannian $\Gr(3,\mathbb{K}^6)$ is identified with the quotient of $Mat_{\mathbb{K}}(3,6)$ by the left-multiplication action of $\operatorname{GL}_3$, i.e. 
\begin{equation*}  \Gr(3,\mathbb{K}^6) =  \operatorname{GL} _3 \backslash Mat_{\mathbb{K}}(3,6) \end{equation*}   
Now let $\Gr_0(3,\mathbb{K}^6)$ be the points in $\Gr(3,\mathbb{K}^6)$ with a representative in $Mat_{\mathbb{K}}(3,\mathbb{K}^6)$ whose maximal minors do not vanish (in fact, this will hold for any representative). The torus acts on $\Gr_0(3,\mathbb{K}^6)$. The action on $Mat_{\mathbb{K}}(3,6)$ by right multiplication of diagonal invertible $6\times 6$ matrices induces an action of the torus $\mathbb{G}_m^6(\mathbb{K})$ on $Gr_0(3,\mathbb{K}^6)$. The Gelfand-MacPherson correspondence (see e.g. \cite[Section 2.6]{KeelEtAl}) provides the identification 
\begin{equation*} X^6 = \Gr_0(3,\mathbb{K}^6)/\mathbb{G}_m^6(\mathbb{K}). \end{equation*}
Here, we view the columns of the matrix representative as the points in the configuration. The Pl\"ucker embedding induces an embedding of $X^6$ into the torus $\mathbb{G}_m^{20}(\mathbb{K})/\mathbb{G}_m^{6}(\mathbb{K}) \cong \mathbb{G}_m^{14}$ as a closed subvariety. Under this correspondence, the 6 points in $\mathbb{P}^2(\mathbb{K})$ lie on a conic if and only if the Pl\"ucker coordinates satisfy
\begin{equation*} C := p_{134}p_{156}p_{235}p_{246} - p_{135}p_{146}p_{234}p_{256} = 0.  \end{equation*}
To see this, note that there is only one conic up to projective transformation, e.g. take $xz = y^2$. So the points lie on a conic if and only if this configuration can be represented by a matrix of the form 
\begin{equation*}
\left[\begin{array}{cccccc}
1 & 1 & 1 & 1 & 1 & 1 \\
a_1 & a_2 & a_3 & a_4 & a_5 & a_6 \\
a_1^2 & a_2^2 & a_3^2 & a_4^2 & a_5^2 & a_6^2 \\
\end{array} \right]
\end{equation*}

for $a_1, \ldots, a_6$ in $\Bbbk$. It suffices to check the above Pl\"ucker identity for this matrix. The vanishing locus of $C$ corresponds to an irreducible Weil divisor of $X^6$. By the description of degree 3 del Pezzo surfaces as blow-ups of $\mathbb{P}^2$ at 6 points in general position, we may identify $Y^6$ with $X^6 \setminus V(C)$. Under this identification, we see that $Y^6$ is a very affine variety and can be realized as a closed subvariety of $\mathbb{G}_m^{15}$ (this follows from \cite[Lemma 6.1]{HackingEtAl}). Therefore, the tropicalization of $Y^6$ may be viewed as the underlying set of a pure $4$ dimensional fan in $\mathbb{R}^{15}$. By \cite{HackingEtAl}, $\trop(Y^6)$ admits a unique coarsest fan structure called the Naruki fan. The coordinates of this fan compute the possible valuations of the Pl\"ucker coordinates $p_{ijk}$ and $C$ (up to the action of $\mathbb{G}_m^6$).

\section{The search for a combinatorial structure to classify Khovanskii bases for degree~3 del Pezzo surfaces}\label{section_examples}

In order to classify 3-dimensional Khovanskii subspaces of $\Bbbk^6$ we are looking for a combinatorial structure which parametrizes equivalence classes of such subspaces. When we identify a right structure (probably a fan of convex polyhedral cones), the next, and the last, step will be to subdivide it such that each chamber in the subdivision corresponds to a different class of moneric bases, some of them Khovanskii.

\subsection{Degree~4 del Pezzo surfaces and $\operatorname{TGr}(2,\mathbb{Q}^5)$}
In the case of the Cox ring of del Pezzo surface of degree~4, i.e.~$G$ being represented by a $2\times 5$ matrix, this role was played by the tropical Grassmannian $\operatorname{TGr}(2,\mathbb{Q}^5)$, introduced in Section~\ref{section_trop}. It is a 7-dimensional fan in the 10-dimensional space, a product of a 5-dimensional lineality space and the cone over the Petersen graph. It is worth noting that this 2-dimensional part is also the tropicalization of (the very affine part of) the moduli space of degree 4 del Pezzo surfaces, see~\cite{RenEtAl}.

The map from the set of equivalence classes of subspaces~$G$ to $\operatorname{TGr}(2,\mathbb{Q}^5)$ is given by the tropical Pl\"ucker coordinates $d_{ij} = -\operatorname{val}(p_{ij})$ for $1\leq i < j \leq 5$. In this way $\operatorname{TGr}(2,\mathbb{Q}^5)$, or its set of integral points, becomes a \emph{good parametrizing set} for equivalence classes of moneric and Khovanskii subspaces (see Definition~\ref{def_equivalence}). This means that it satisfies the conditions of the following important definition.

\begin{definition}\label{def_parameter_space}
  For a set~$M$ to be a \emph{good parametrizing set} for moneric and Khovanskii subspaces of $\Bbbk^n$ we require that for any subspaces~$G$ and~$G'$ mapped to the same point of~$M$, if~$G$ is moneric (resp. Khovanskii), then~$G'$ is also moneric (resp. Khovanskii).
\end{definition}

The reason for this property is that all coefficients in generators of the Cox ring in this case (see~\cite[Thm~4.1]{SturmfelsEtAl}) are monomials in Pl\"ucker coordinates. Thus if $G$ and $G'$ have the same sequence $(d_{ij})$ then they determine the same initial forms of all generators. In particular, if one of them is moneric or Khovanskii, then the second one also is, and obviously they are equivalent.

\subsection{Degree~3 del Pezzo surfaces, $\operatorname{TGr}(3,\mathbb{Q}^6)$ and the Naruki fan}
To find a fan parametrizing moneric subspaces $G$ for the case of del Pezzo surfaces of degree~3 (which can be embedded in~$\mathbb{P}^3$ as smooth cubic surfaces), we tested two natural candidates. The first one is the tropical Grassmannian $\operatorname{TGr}(3,\mathbb{Q}^6)$.

\begin{example}\label{example_comp_tgr}
Take the subspace represented by the matrix $G$ written below. Its sequence of (negatives of) tropical Pl\"ucker's coordinates is
$$(d_{ijk}) = (5, 11, 10, 4, 13, 15, 9, 18, 12, 15, 4, 10, 1, 9, 3, 6, 14, 8, 11, 14).$$
We modify $G$ slightly to the matrix $G'$ by changing the sign of the fourth term in the first row.

$$
G = \left[\begin{array}{cccccc}
      t^4 & t & t^8 & t^3 & t^9 & 1\\
      t^{11} & t^7 & t & t^7 & t^6 & 1\\
      t^9 & 1 & t^5 & t^9 & t^{11} & t^6
\end{array} \right] \qquad \qquad  G' = \left[\begin{array}{cccccc}
      t^4 & t & t^8 & -t^3 & t^9 & 1\\
      t^{11} & t^7 & t & t^7 & t^6 & 1\\
      t^9 & 1 & t^5 & t^9 & t^{11} & t^6
\end{array} \right]
$$

One checks that the modification does not affect the tropical Pl\"ucker coordinates. That is, $G$ and $G'$ are mapped to the same point of $\operatorname{TGr}(3,\mathbb{Q}^6)$. Thus, if a coefficient in the formula for a generator is a monomial in Pl\"ucker coordinates, it will also take the same value for $G$ and $G'$. Recall that all coefficients of generators corresponding to lines, and also some coefficients of generators corresponding to conics, have this form (see Section~\ref{sect-cox-delPezzo}).

However, generators corresponding to conics have also some coefficients which are binomials in Pl\"ucker coordinates, and it turns out that they are the reason for $\operatorname{TGr}(3,\mathbb{Q}^6)$ being insufficient for our task. Look at the generator corresponding to the conic $G_6$ through points 1, 2, 3, 4, 5 (in~\cite[p.~443]{SturmfelsEtAl}):
\begin{align*}
  G_6 &= p_{123}p_{124}p_{125}p_{345}y_1y_2x_3x_4x_5x_6^2 + p_{123}p_{135}p_{134}p_{245}y_1y_3x_2x_4x_5x_6^2\\
  &+ p_{124}p_{134}p_{145}p_{235}y_1y_4x_2x_3x_5x_6^2 + p_{125}p_{135}p_{145}p_{234}y_1y_5x_2x_3x_4x_6^2\\
  &+ p_{123}p_{234}p_{235}p_{145}y_2y_3x_1x_4x_5x_6^2 + p_{124}p_{234}p_{245}p_{135}y_2y_4x_1x_3x_5x_6^2\\
  &+ p_{125}p_{235}p_{245}p_{134}y_2y_5x_1x_3x_4x_6^2 + p_{134}p_{234}p_{345}p_{125}y_3y_4x_1x_2x_5x_6^2\\
  &+ p_{135}p_{235}p_{345}p_{124}y_3y_5x_1x_2x_4x_6^2 + p_{145}p_{245}p_{345}p_{123}y_4y_5x_1x_2x_3x_6^2\\
  &+ (p_{124}p_{235}p_{136}p_{145} - p_{123}p_{245}p_{146}p_{135})\cdot y_1y_6x_2x_3x_4x_5x_6\\
  &+ (p_{124}p_{135}p_{236}p_{245} - p_{123}p_{145}p_{246}p_{235})\cdot y_2y_6x_1x_3x_4x_5x_6\\
  &+ (p_{134}p_{125}p_{236}p_{345} + p_{123}p_{145}p_{346}p_{235})\cdot y_3y_6x_1x_2x_4x_5x_6\\
  &+ (p_{124}p_{135}p_{346}p_{245} - p_{134}p_{125}p_{246}p_{345})\cdot y_4y_6x_1x_2x_3x_5x_6\\
  &+ (p_{125}p_{134}p_{356}p_{245} - p_{135}p_{124}p_{256}p_{345})\cdot y_5y_6x_1x_2x_3x_4x_6\\
  &+ (p_{124}p_{135}p_{236}p_{456} - p_{123}p_{145}p_{246}p_{356})\cdot y_6^2x_1x_2x_3x_4x_5.
\end{align*}

Note that the signs are different that in~\cite{SturmfelsEtAl}, which is a result of permuting the indices in Pl\"ucker coordinates.

We compute the valuation of its second binomial coefficient, $p_{124}p_{135}p_{236}p_{245} - p_{123}p_{145}p_{246}p_{235}$: for $G$ it is 36, but for $G'$ it is 37. Both monomials have valuation 36, as shown by the sequence $(d_{ijk})$, but for $G'$ the coefficients are such that the lowest terms cancel.

Moreover, computation of the remaining coefficients for $G_6$ show that in both cases this is the minimal valuation. Only for $G$ it is the smallest one, and for $G'$ there are more coefficients with valuation 37. We obtain that initial forms of $G_6$ are
$$ -2x_1x_3x_4x_5x_6y_2y_6 \quad \hbox{and} \quad -x_1x_4x_5(x_6^2y_2y_3+2x_3x_6y_2y_6+x_2x_3y_6^2)$$
for $G$ and $G'$ respectively. That is, $G'$ is \emph{not moneric}, and one can check by computing other generators that~$G$ is. This example was constructed using \emph{Macaulay2}~\cite{M2}.
\end{example}

Thus we have two subspaces mapped to a single point of $\operatorname{TGr}(3,\mathbb{Q}^6)$, such that one is moneric and the second is not. This shows that the tropical Grassmannian is too coarse to be a good parametrizing set: the property of being moneric is not well-defined for its points. It is worth noting that the point corresponding to~$G$ and~$G'$ does not lie in the interior of a maximal cone of $\operatorname{TGr}(3,\mathbb{Q}^6)$, but we expect the same phenomenon to appear also at interior points of maximal cones.

The second candidate for the parametrizing space is, as suggested in~\cite[Prob.~5.4]{SturmfelsEtAl}, the tropical moduli space of (smooth, marked) del Pezzo surfaces of degree~3. Its combinatorial structure is the Naruki fan, described in Section~\ref{section_trop} (see also~\cite{HackingEtAl} and~\cite[Sect.~6]{RenSamSturmfels}). Recall that a $3\times 6$ matrix~$G$ corresponds to a sequence of coordinates which are either monomials in Pl\"ucker coordinates or products of such monomials and a binomial $C = p_{134}p_{156}p_{235}p_{246} - p_{135}p_{146}p_{234}p_{256}$, which encodes the condition for 6 points lying on a conic.

\begin{example}\label{example_comp_naruki}
We compute the value of the binomial $C$ for both $G$ and $G'$ and obtain the result that they both have valuation 37. This means that these matrices are mapped to the same point in $\mathrm{trop}(Y^6)$, hence this is not a good space for parametrizing moneric classes.
\end{example}

To summarize, Examples~\ref{example_comp_tgr} and~\ref{example_comp_naruki} prove the following result.
\begin{proposition}
Neither the tropical Grassmannian $\operatorname{TGr}(3,\mathbb{Q}^6)$ nor the Naruki fan is a good parametrizing set for moneric classes of 3-dimensional subspaces of~$\Bbbk^6$ in the sense of Definition~\ref{def_parameter_space}.
\end{proposition}

The conclusion is that to find a good parametrizing set for our problem we should probably look for an another variety (maybe a different embedding of $Y^6$), whose coordinates are more closely related to binomials which appear in the Cox ring conic generators. Of 36 binomials appearing in 6 conic generators, 6 are equivalent (up to a Pl\"ucker relation) to $C$, and the remaining 30 are different, and also pairwise different. Hence our strategy will be to consider a variety embedded in a projective space using all~31 equivalence classes of binomials, tropicalize it and subdivide the fan structure obtained in this way to parametrize 3-dimensional moneric and Khovanskii subspaces of~$\Bbbk^6$.

We finish with a remark that the tropical moduli space of cubic surfaces is not sufficient for one more reason: it requires being enlarged by adding a lineality space to the fan.

\begin{example}
  Consider
    $$G'' = \left[\begin{array}{cccccc}
      1 & t & t^8 & t^3 & t^9 & 1\\
      t^{7} & t^7 & t & t^7 & t^6 & 1\\
      t^{5} & 1 & t^5 & t^9 & t^{11} & t^6
\end{array} \right]$$
  which comes from $G$ by multiplying the first column by $1/t^4$. Note that if we treat columns of a matrix as coordinates of points in $\mathbb{P}^2$, then $G$ and $G''$ represent the same choice of 6 points, so the same marked del Pezzo surface. Note also that if we looked at the kernels of $G$ and $G''$ as choices of 6 points in $\mathbb{P}^2$, we would also get the same sets, because multiplying the first column of $G$ by $1/t^4$ corresponds to multiplying the first row of a matrix representing $\operatorname{ker} G$ by $t^4$. However, one can compute the conic generator $G_6$ for $G''$ and learn that it has a binomial leading term, so $G''$ is not moneric. 
\end{example}

This shows that the property of being moneric is not well-defined for a marked del Pezzo surface -- its behaviour varies in the set of matrices representing the same choice of 6 points on the plane. The same phenomenon can be observed also in the case of degree~4 del Pezzo surfaces, where $\operatorname{TGr}(2,\mathbb{Q}^5)$ was used to parametrize moneric subspaces. This is one of the reasons for considering the full $\operatorname{TGr}(2,\mathbb{Q}^5)$, not only the tropicalization of the moduli space of degree~4 del Pezzo surfaces, i.e. the cone over the Petersen graph. The lineality space is equally important. The subdivision determining equivalence classes of moneric subspaces is not a pull-back of a subdivision of the cone over the Petersen graph to $\operatorname{TGr}(2,\mathbb{Q}^5)$ via the projection along the lineality space, it cuts through fibers of this projection. Thus, by analogy, we expect that to use some variant of the tropical moduli space of del Pezzo surfaces of degree~3 to parametrize 3-dimensional moneric subspaces of~$\Bbbk^6$ we should also enlarge it by adding a lineality space.

\begin{acknowledgement}
This article was initiated during the Apprenticeship Weeks (22 August-2 September 2016), led by Bernd Sturmfels, as part of the Combinatorial Algebraic Geometry Semester at the Fields Institute. The authors are very grateful to Bernd Sturmfels for suggesting the problem, discussions and encouragement.
Daniel Corey was supported by NSF CAREER DMS-1149054. Maria Donten-Bury was supported by a Polish National Science Center project 2013/11/D/ST1/02580. Naoki Fujita was supported by Grant-in-Aid for JSPS Fellows (No.~16J00420).
\end{acknowledgement}

\end{document}